\documentclass[final,oneeqnum,onethmnum,onefignum]{siamltex}
\usepackage{epsfig}
\usepackage{epstopdf}  
\usepackage{amsfonts}

\newcommand {\onehalf}{1/2}

\newcommand {\onetwelfth}{1/12}
\newcommand {\erfcinv}{\mathrm{erfc}^{-1}}
\newcommand {\floor}[1]{{\left\lfloor{#1}\right\rfloor}}
\newcommand {\abs}[1]{{\left|{#1}\right|}}

\newcommand{\ekc}{E(k_c)}

\newcommand{\ekcstar}{E(k_c^\mathrm{star})}
\newcommand{\kcstar}{k_c^\mathrm{star}}
\newcommand{\ekcchain}{E(k_c^\mathrm{chain})}
\newcommand{\kcchain}{k_c^\mathrm{chain}}
\newcommand{\ekcdb}{E(k_c^\mathrm{db})}
\newcommand{\kcdb}{k_c^\mathrm{db}}

\title{Synchronisation Properties of Trees in the Kuramoto Model}

\author{Anthony~H.~Dekker\thanks{Graduate School of Information Technology \& Mathematical Sciences, University of Ballarat, Australia (\texttt{dekker@acm.org}). Part of this work was conducted while in the employ of DSTO.} \and Richard Taylor\thanks{Joint Operations Division, DSTO, Canberra, ACT (\texttt{richard.taylor@dsto.defence.gov.au})}}

\begin{document}

\maketitle

\begin{abstract}
We consider the Kuramoto model of coupled oscillators, specifically the case of tree networks, for which we prove a simple closed-form expression for the critical coupling. For several classes of tree, and for both uniform and Gaussian vertex frequency distributions, we provide tight closed form bounds and empirical expressions for the expected value of the critical coupling. We also provide several bounds on the expected value of the critical coupling for all trees. Finally, we show that for a given set of vertex frequencies, there is a rearrangement of oscillator frequencies for which the critical coupling is bounded by the spread of frequencies.
\end{abstract}



\pagestyle{myheadings}
\thispagestyle{plain}
\markboth{A.~H. DEKKER AND R. TAYLOR}{SYNCHRONISATION PROPERTIES OF TREES}


\section{Introduction}
\label{sec:intro}

The Kuramoto model \cite{Arenas2008, Bronski2012, Dorogovtsev2008, Kalloniatis2010, Strogatz1998, Strogatz2000, Strogatz2003} was originally motivated by the phenomenon of collective synchronisation whereby a system of coupled oscillating vertices (nodes) will sometimes lock on to a common frequency despite differences in the natural frequencies of the individual vertices. Biological examples include oscillations of the heart \cite{Winfree1980} and of chemical systems \cite{Kuramoto1984}. Systems of coupled oscillators can also be used as an abstract model for synchronisation in organisations \cite{Dekker2007,  Dekker2010, Kalloniatis2008}.
While Kuramoto studied the infinite complete network, it is natural to consider finite networks of any given topology. This would correspond to a notion of coupling that is not universal across all vertex (node) pairs, but rather applies to a subset of all possible edges (links). For example the work patterns of human individuals in an organisation might enjoy a coupling effect in relation to pairs of individuals that have a working relationship.

A typical network of coupled oscillators is as shown in Figure \ref{fig:intro}. Each vertex has an associated phase angle $\theta_i$, as well as its own natural frequency $\omega_i$. In this paper we are primarily concerned with the case where the network is a tree. The basic governing equation is the differential equation:
\begin{equation}
\dot{\theta}_i = \omega_i + k \sum_{i=1}^n A_{ij} \sin (\theta_j - \theta_i)
\end{equation}
where $A$ is the adjacency matrix of the network and $k$ is a coupling constant which determines the strength of the coupling.  We refer to a network (graph) with preferred
oscillator frequencies $\omega_i$ attached to the vertices as a \textit{Kuramoto graph}.
 
\begin{figure}[htbp]
\begin{center}
\includegraphics[width=7.8cm]{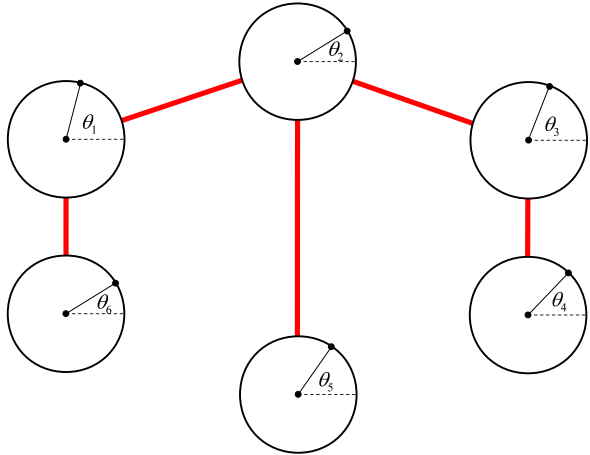}
\caption{An example Kuramoto tree.} \label{fig:intro}
\end{center}
\end{figure}

It has been observed \cite{Dorogovtsev2008, VanHemmen1993} that many Kuramoto graphs \textit{synchronise}, in that the actual vertex frequencies $\dot{\theta}_i $ converge to a common value. That is, the vertex phases $\theta_i$ rotate at the same rate, with a constant phase difference between each pair of vertices. Moreover this phenomenon appears at a \textit{critical coupling constant} $k_c$, and as $k$ increases, this phenomenon applies to a greater range of natural frequencies $\omega_i$ (in \cite{Strogatz1998}, the probability of synchronisation is studied). Thus the graph has a \textit{frequency fixed point} characterised by all the $\dot{\theta_i}$ being equal.

At the frequency fixed point, the $\dot{\theta_i}$ are equal to $\bar{\omega}$, the mean of the frequencies $\omega_i$, and it is convenient to apply a rotating frame of reference, with:
\begin{equation}
\phi_i (t) = \theta_i (t) - \bar{\omega} t
\end{equation}
We then obtain a system of differential equations equivalent to the original:
\begin{equation}\label{rotating:eq}
\dot{\phi}_i = \omega_i - \bar{\omega} + k \sum_{i=1}^n A_{ij} \sin (\phi_j - \phi_i)
\end{equation}
At the frequency fixed point, we then have $\dot{\phi_i} = 0$ for all $i$. While
the literature contains several definitions of ``critical coupling,'' here we
define the critical coupling $k_c$ as that number for which a frequency fixed point exists exactly when $k \geq k_c$ (in \cite{Jadbabai2004}, this critical coupling is called $K_L$).

In general the value of the critical coupling is the solution of simultaneous transcendental equations and would not be expected to have a closed form solution. However, for complete graphs and complete bipartite graphs the critical coupling can be computed efficiently 
as the solution of non-linear equations \cite{Verwoerd2008, Verwoerd2009}. For the graph with only two vertices $\nu_1$ and $\nu_2$ and a single edge between them, it is easy to see that solutions exist precisely when $\abs{\omega_1 - \omega_2} \leq 2k$, i.e.
\begin{equation}\label{twin:eq}
k_c = \frac{\abs{\omega_1 - \omega_2}}{2}
\end{equation}
Our first theorem generalises this to any tree. We then use that result to calculate the \textit{expected value} of the critical coupling for several classes of Kuramoto tree having randomly chosen frequencies, and to find bounds on the expected value of the critical coupling for Kuramoto trees in general. Finally, we use the result to show that the critical coupling for a specific tree topology can be reduced to a value which is independent of the number of vertices, by reshuffling the frequencies $\omega_i$.


\section{Main Lemma and Corollaries}
\label{sec:lemmas}

The first step to our initial theorem is a lemma which simplifies the calculation of critical
couplings by breaking a Kuramoto graph into components. In particular, we consider
Kuramoto graphs with a cut-vertex, i.e.\ one whose removal disconnects the graph. 
Figure \ref{rtfigure1} shows an example.

\begin{figure}[htbp]
\begin{center}
\includegraphics[width=7.8cm]{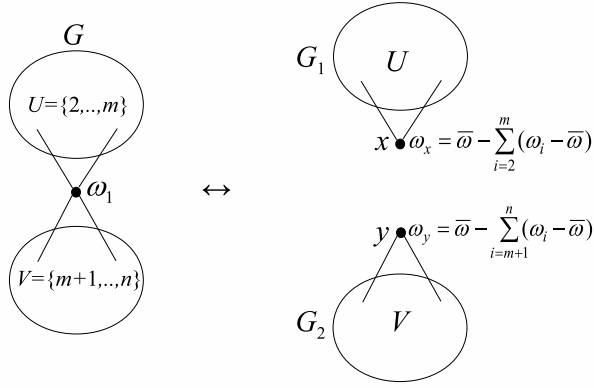}
\caption{Construction for Lemma \ref{main:lemma}.}
\label{rtfigure1}
\end{center}
\end{figure}

\begin{lemma}\label{main:lemma}
Let $G$ be a Kuramoto graph with $n$ vertices $\nu_1,\nu_2,\ldots,\nu_n$ with natural frequencies $\omega_{1},\omega_{2},\ldots,\omega_{n}$ in which vertex $\nu_1$ is a cut-vertex (whose removal disconnects the graph) with vertices $\nu_2,\ldots,\nu_m$ on one side of the cut and vertices $\nu_{m+1},\ldots,\nu_n$ on the other. As shown in Figure \ref{rtfigure1}, let $G_{1}$ and  $G_{2}$ be the two graphs formed by deleting all the vertices on one side of the cut vertex and replacing the cut vertex by new vertices $x$ (for $G_{1}$) and $y$ (for $G_{2}$) having frequencies $\omega_{x}$ and $\omega_{y}$:
\begin{equation}
\omega_{x} = \bar{\omega} - \sum_{i=2}^m (\omega_i - \bar{\omega}) = \omega_1 + \sum_{i=m+1}^n (\omega_i - \bar{\omega}) \label{eq:omegax}
\end{equation}
\begin{equation}
\omega_{y} = \bar{\omega} - \sum_{i=m+1}^n (\omega_i - \bar{\omega}) = \omega_1 + \sum_{i=2}^m (\omega_i - \bar{\omega}) \label{eq:omegay}
\end{equation}
Then $G$ has a frequency fixed point if and only if both $G_1$ and $G_2$ have frequency fixed points, and hence $k_c (G) = \max(k_c(G_1), k_c(G_2))$. 
\end{lemma}

\begin{proof}
First we note that the frequencies $\omega_{x}$ and $\omega_{y}$ are chosen in \ref{eq:omegax} and \ref{eq:omegay} so that the average frequencies of both $G_{1}$ and $G_{2}$ are the same as that of $G$.  Let $G$ have a frequency fixed point so that we have $k$ and $\phi_{i}, i=1,\ldots,n$ satisfying the system of equations for $i=1,\ldots,n$:
\begin{equation}\label{rt:eq1}
0=\omega_{i}-\bar{\omega}+k\sum_{j=1}^n A_{ij}\sin(\phi_{j}-\phi_{i})
\end{equation}
If we sum these equations from $i=1,\ldots,m$, then each term $A_{ij}\sin(\phi_{j}-\phi_{i})$ is cancelled by the term $A_{ji}\sin(\phi_{i}-\phi_{j})$ with the exception of the pairs $(1,j), j=m+1,\ldots,n$ and we obtain:
\begin{equation}\label{rt:eq2}
0=\sum_{i=1}^m (\omega_{i}-\bar{\omega})+k\sum_{i=m+1}^n A_{1i}\sin(\phi_{i}-\phi_{1})
\end{equation}
Subtracting equation (\ref{rt:eq2}) from (\ref{rt:eq1}), for the case $i=1$, gives an equation for vertex $x$:
\begin{eqnarray}
0&=&-\sum_{i=2}^m (\omega_{i}-\bar{\omega})+k\sum_{i=2}^m A_{1i}\sin(\phi_{i}-\phi_{1}) \nonumber\\ &=&\bar{\omega}-\omega_{x}+k\sum_{i=2}^m A_{1i}\sin(\phi_{i}-\phi_{1})
\end{eqnarray}
Similarly by summing (\ref{rt:eq1}) from $i=m+1,\ldots,n$ we obtain an equation for vertex $y$:
\begin{equation}
0=\bar{\omega}-\omega_{y}+k\sum_{i=m+1}^n A_{1i}\sin(\phi_{i}-\phi_{1})
\end{equation}
It follows that a solution to the system of equations (\ref{rt:eq1}) is a solution to the system of equations for $G_{1}$ and $G_{2}$ in which the vertices $x$ of $G_{1}$ and $y$ of $G_{2}$ have the same phase $\phi_{1}$. If on the other hand we have a solution to the system of fixed point equations for $G_{1}$ and $G_{2}$, then we can shift all the phase angles of the vertices of $G_{2}$ by the same amount so that $\phi_{y}=\phi_{x}$, without changing the phase angle differences. This provides a fixed point solution to the system of equations (\ref{rt:eq1}) and completes the proof. \qquad
\end{proof}

Repeated application of this lemma can reduce the calculation of the critical coupling to 2-connected graphs (in which any two vertices lie on a cycle) and tree edges. For example, we have the following corollary for graphs with a cut-edge, i.e.\ an edge whose removal disconnects the graph:

\begin{corollary}\label{bell:corollary}
Let $G$ be a Kuramoto graph with $n$ vertices $1,2,\ldots,n$ with natural frequencies $\omega_{1},\omega_{2},\ldots,\omega_{n}$ in which the edge from $\nu_1$ to $\nu_2$ is a cut-edge with vertices $\nu_1$ and $\nu_3,\ldots,\nu_m$ on one side of the cut and vertices $\nu_2$ and $\nu_{m+1},\ldots,\nu_n$ on the other. Let $G'$ be the one-edge graph containing only the vertices $x$ and $y$, where:
\begin{eqnarray}
\omega_{x} & = & \omega_1 + \sum_{i=3}^m (\omega_i - \bar{\omega}) \\
\omega_{y} & = & \omega_2 + \sum_{i=m+1}^n (\omega_i - \bar{\omega})
\end{eqnarray}
Then $G$ has a frequency fixed point only if $G'$ has a frequency fixed point, i.e.
\begin{equation}
\abs{\omega_1 - \omega_2 + \sum_{i=3}^m (\omega_i - \bar{\omega}) - \sum_{i=m+1}^n (\omega_i - \bar{\omega})} \leq 2 k
\end{equation}
\end{corollary}

\begin{proof}
Applying Lemma \ref{main:lemma} twice yields three components: one based on the vertices $\{\nu_1, \nu_3,\ldots,\nu_m\}$, one based on $\{\nu_2, \nu_{m+1},\ldots,\nu_n\}$, and  the one-edge graph $G' = \{x,y\}$ based on $\{\nu_1, \nu_2\}$. The condition on $k$ follows by applying (\ref{twin:eq}) to $G'$. \qquad
\end{proof}

For graphs which are trees, repeated application of Lemma \ref{main:lemma} produces components which are exactly the tree edges, with a similar adjustment of frequencies. We can therefore obtain an ``if and only if'' condition, and hence obtain a precise value for $k_c$:

\begin{theorem}\label{main:theorem}
Let $T$ be any tree with at least two vertices. For any edge $e$ of $T$, let $T_{e}$ and $T'_{e}$ be the two subtrees formed when $e$ is deleted from $T$. These two partitions have $|T_{e}|$ and $|T'_{e}|$ vertices respectively. Define the partition sum:
\begin{equation}
\Omega(T_{e})=\sum_{i \in T_{e}}|\omega_{i}-\bar{\omega}|=\left||T_{e}|\bar{\omega}-\sum_{i \in T_{e}}\omega_{i}\right|
\end{equation}
to be the sum of the frequency deviations from $\bar{\omega}$. Then $\Omega(T_{e})=\Omega(T'_{e})$ and the critical coupling $k_{c}$ for $T$ is given by:
\begin{equation}
k_c = \max_{ e \in T}\Omega(T_{e})      \label{maineq:15}
\end{equation}
\end{theorem}

\begin{proof}
We repeatedly apply Lemma \ref{main:lemma} to obtain $n-1$ graphs each consisting of a single edge. The condition on $k$ follows from (\ref{twin:eq}), applied to each of the $n-1$ edge graphs obtained. \qquad
\end{proof}

We note that, for trees, the condition in this theorem is equivalent to the more general sufficient conditions given in Remark 10 of \cite{Jadbabai2004} and Statement G1 of \cite{Dorfler2012}.  In the remainder of the paper we consider various implications of the result, facilitated by the form in which our result is expressed. We begin by considering the expected value of $k_c$, when frequencies are chosen randomly.


\section{Expected Critical Couplings}
\label{sec:exp}

It may be the case that we do not know the precise frequencies for a Kuramoto tree,
but we do know the \textit{distribution} of the frequencies.  In this case, we can apply Theorem \ref{main:theorem} to determine the \textit{expected value} of the critical coupling $k_c$ for various kinds of tree. Figure \ref{fig:examples} shows some of the trees considered.
In this section we combine theoretical derivations, using Theorem \ref{main:theorem} and statistical theory, with Monte Carlo experiments. In these experiments, we considered a variety of trees, and made 1,000,000 frequency assignments with $\omega_i$ uniformly distributed over the interval [0,1]. The frequencies $\omega_i$ therefore have mean \onehalf\ and variance $\sigma^2 = \onetwelfth$, and the expected critical coupling $\ekc$ is proportional to $\sigma$, the square root of the variance. We used Theorem \ref{main:theorem} to determine the critical coupling $k_c$ for each assignment, and hence the expected value of the critical coupling for that size and shape of tree. We also performed the same experiments using normally distributed frequencies with the same mean and variance, calculating over 8,000,000 frequency assignments in that case, to allow adequate sampling of the tail.

\begin{figure}[htbp]
\begin{center}
\begin{tabular}{|c|c|}
\hline
\includegraphics[width=3.4cm]{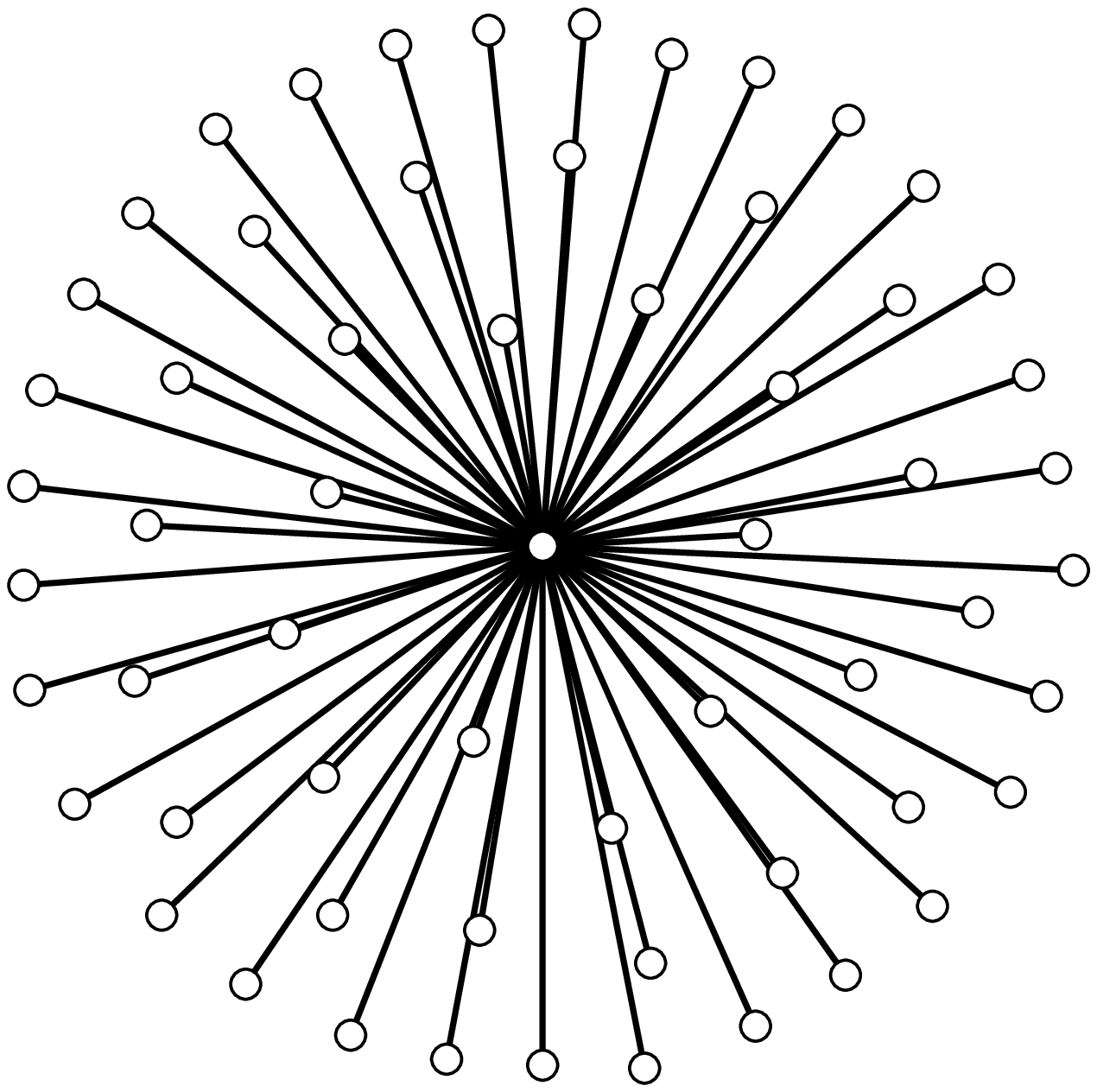} & \includegraphics[width=3.4cm]{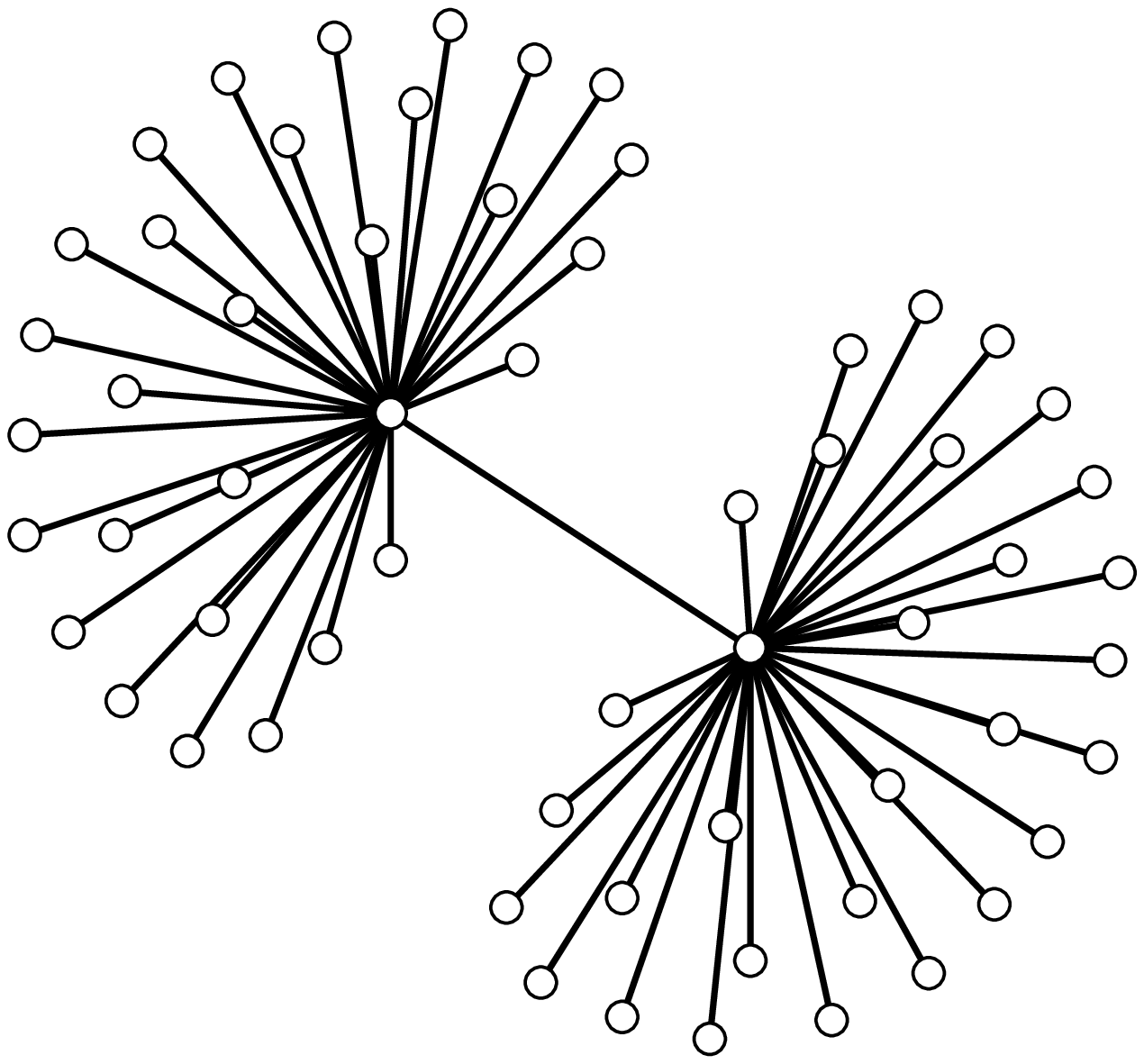} \\
(a) Star (asterisk) & (b) Dumb-bell \\
\hline
\includegraphics[width=3.4cm]{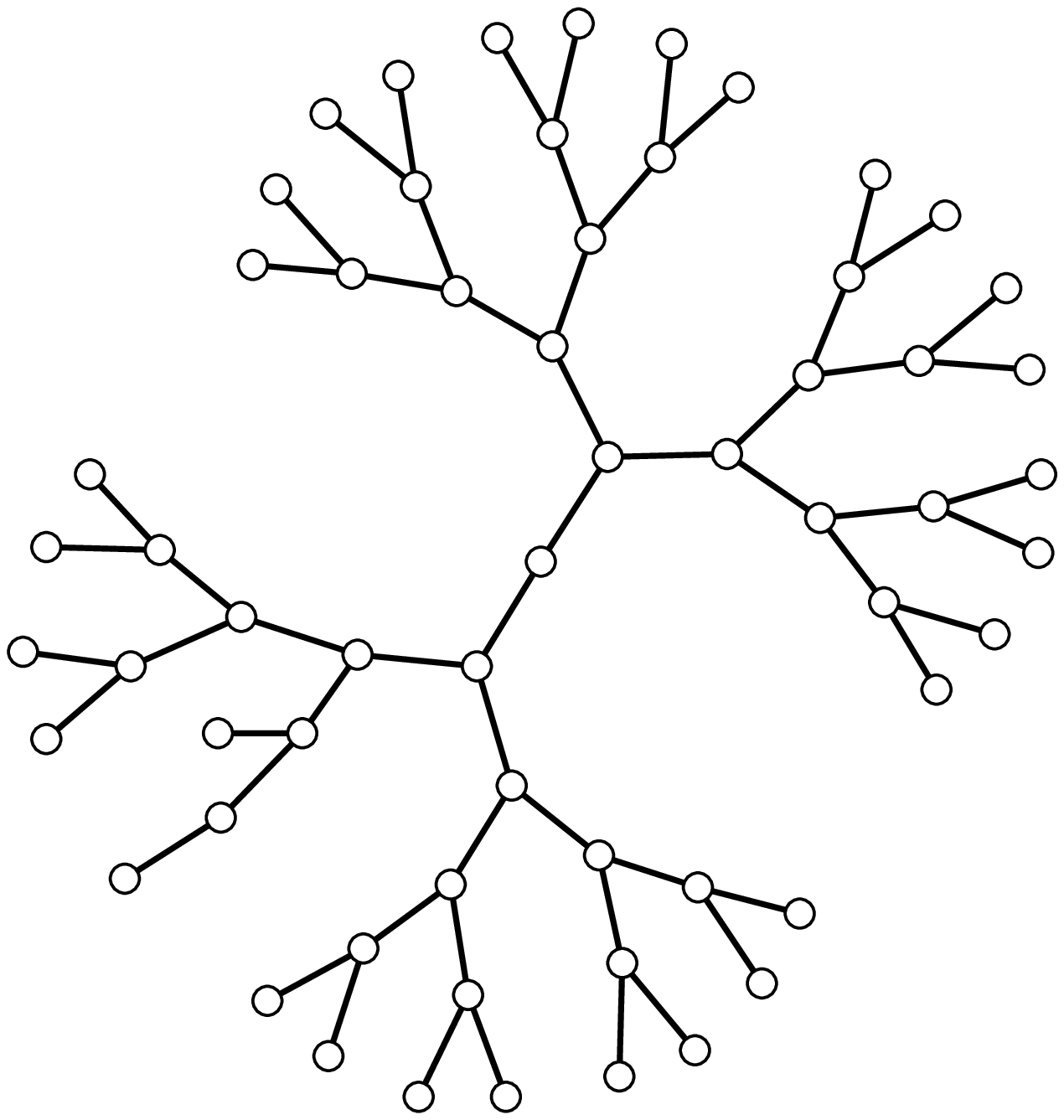} & \includegraphics[width=3.4cm]{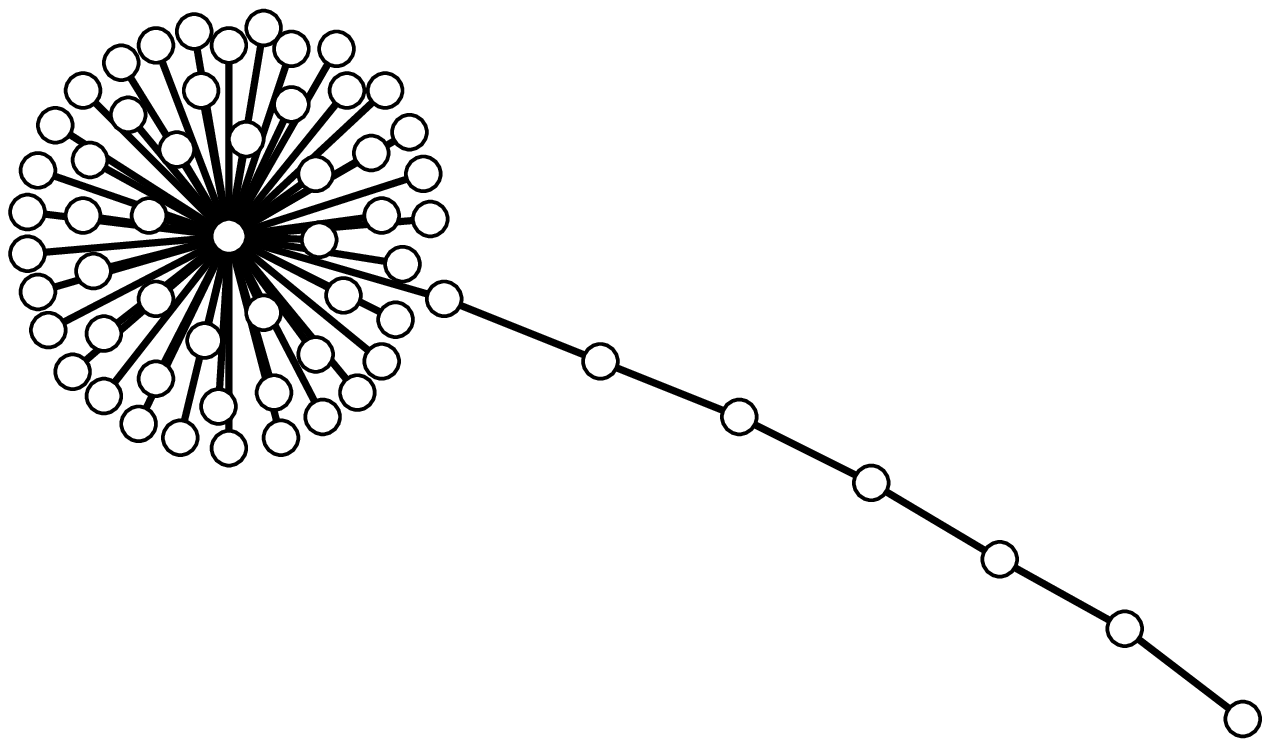} \\
(c) Binary tree & (d) ``Tadpole'' \\
\hline
\end{tabular}
\caption{Some example trees with $n = 60$.} \label{fig:examples}
\end{center}
\end{figure}

\subsection{Chains}

For the case of a chain (a tree with all vertices connected in a line) with $n$ vertices $\nu_1, \ldots, \nu_n$,
by Theorem \ref{main:theorem}, the critical coupling $\kcchain$ is:
\begin{equation}
\kcchain = \max_j \abs{ \sum_{i=1}^{j} (\omega_i - \bar{\omega}) } \label{neweq:16}
\end{equation}
The partition sums $\sum_{i=1}^{j} (\omega_i - \bar{\omega})$ for the first $j$ vertices form a random walk as $j$ increases. Figure \ref{fig:walk} illustrates an example, for $n = 60$. 
The walk must return to zero, since $\bar{\omega} = \sum_{i=1}^{n} \omega_i $, and so the maximum translation distance is likely to occur in the vicinity of the centre of the chain.  This random walk has step variance $\sigma^2$, and a crude estimate for the expected critical coupling $\ekcchain$ is the expected final displacement $\Delta$ for a random walk with this step variance and $\tau$ steps, where $\tau$ is some value between $n/2$ and $n$.  However, this takes into account neither the return to zero nor the fact that the maximum in (\ref{neweq:16}) is being taken over all values of $j$ (these two properties mean that, as pointed out in \cite{Strogatz1998}, we are dealing with ``pinned Brownian motion,'' also known as a ``Brownian bridge'').

\begin{figure}[htbp]
\begin{center}
\includegraphics[width=7.8cm]{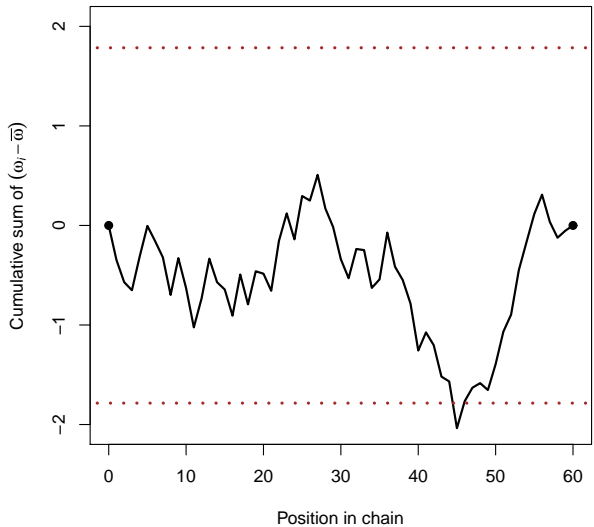}
\caption{One random walk of partition sums $\sum_{i=1}^{j} (\omega_i - \bar{\omega})$ for a chain with $n = 60$ and increasing $j$. Dotted lines show the empirical expected value $\ekcchain = \chi(n)$, which is the expected peak deviation from 0.} \label{fig:walk}
\end{center}
\end{figure}

An asymptotic expression for $\Delta$ was given by Coffman \textit{et al.} \cite{Coffman1998} and further studied by Comtet and Majumdar \cite{Comtet2005}.  For the case $\tau = n/2$, this expression gives:
\begin{equation}
\Delta = \sigma \sqrt{\frac{2}{\pi} \frac{n}{2}} - \sigma c = \sigma \sqrt{\frac{n}{\pi}} - \sigma c = \sqrt{\frac{n}{12 \pi}} - 0.148976
\end{equation}
where $c = 0.516068$, and we have ignored terms in $n^{-1/2}$.  Although this expected translation distance under-estimates the critical coupling, we will still have a result of the form $a \sqrt{n} + b$ (from \cite[p. 158]{Strogatz1998}, the critical coupling is $\mathcal{O}(\sqrt{n})$), and we can determine the values
of $a$ and $b$ empirically.

For our empirical Monte Carlo studies, we considered chains with the number of vertices $n$ ranging from 2 to 300,
with 1,000,000 different uniformly distributed frequency assignments (and 8,000,000 different normally distributed frequency assignments)
on each. We used Theorem \ref{main:theorem} to calculate critical couplings in each case, and fitted a curve of the form $a \sqrt{n} + b$ to the expected values.
Empirically, as illustrated in Figure \ref{fig:chain}, the expected critical coupling for a chain of $n$ vertices is:
\begin{equation}
\ekcchain = \chi(n) \approx 0.252 \sqrt{n} - 0.168 = 0.873 \, \sigma \sqrt{n} - 0.581 \, \sigma \label{chain:chidef}
\end{equation}
We express the result both in terms of $\sigma$, and for our special case of $\sigma = \sqrt{1/12}$.
As we will see in \S \ref{boundsec},  the expected critical coupling for the chain acts as an upper bound on $\ekc$ for trees in general. To facilitate use in this way, we abbreviate this function as $\chi(n)$. An upper bound on $\chi(n)$ itself is the expected maximum displacement for $n$ steps given by Weiss \cite[p. 192]{Weiss1994}. This is too high because it ignores the constraint of returning to zero: 
\begin{equation}
\ekcchain = \chi(n) \le \sigma \sqrt{\frac{\pi n}{4}} = \sqrt{\frac{\pi n}{48}}
\end{equation}
Figure \ref{fig:chain} shows this bound, together with the corresponding lower bound, which is the expected maximum displacement for $n/2$ steps. This is too low because it ignores the fact that the peak deviation from zero can occur at either end of the chain:
\begin{equation}
\ekcchain = \chi(n) \ge \sigma \sqrt{\frac{\pi n}{8}} = \sqrt{\frac{\pi n}{96}}
\end{equation}
A better result is obtained by noting that the maximum displacement for ``pinned Brownian motion'' follows Kolmogorov's Distribution \cite{Marsaglia2003}, and the expected maximum displacement therefore asymptotically approaches:
\begin{equation}
\ekcchain = \chi(n) \approx \sigma \sqrt{\frac{\pi n}{2}} \log 2 = 0.869 \, \sigma \sqrt{n} = 0.251 \sqrt{n}
\end{equation}
This asymptotic formula is in close agreement with the empirical formula (\ref{chain:chidef}), and indeed provides an explanation of that formula.  However, as Figure \ref{fig:chain} illustrates, for the finite values of $n$ which we are considering, the asymptotic formula gives an estimate which is a trifle too high.

\begin{figure}[htbp]
\begin{center}
\includegraphics[width=7.8cm]{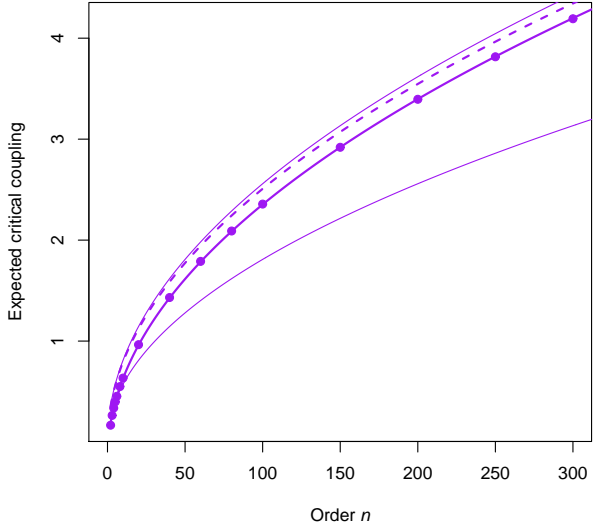}
\caption{Expected critical couplings for the chain, as a function of the number of vertices $n$. The thick line shows the empirical curve $\ekcchain = \chi(n) \approx 0.252 \sqrt{n} - 0.168$, while the thin lines shows the upper and lower random-walk bounds $\sqrt{\pi n / 48}$ and $\sqrt{\pi n / 96}$. The dashed line shows the asymptotic formula for Kolmogorov's Distribution. Results for normally distributed frequencies are, as expected, virtually identical, and cannot be distinguished on this plot. } \label{fig:chain}
\end{center}
\end{figure}

\subsection{Stars}

For star trees, such as the one in Figure \ref{fig:examples} (a), where $\nu_1$ is the central vertex:
\begin{equation}
\kcstar = \max_{i=2\ldots n} \abs{\omega_i - \bar{\omega}}
\end{equation}
For our choice of frequencies, taken from a distribution with mean \onehalf, the critical coupling, which is the expected value of this maximum, is given approximately by:
\begin{equation}
\ekcstar \approx E \left( \max_i \abs{ \omega_i - \frac{1}{2} } \right) + E \left( \abs { \frac{1}{2} - \bar{\omega} } \right) \label{neweq:22}
\end{equation}
where the second term is small. The expected maximum of $\abs{\omega_i - \frac{1}{2}}$, which is uniformly distributed, is given approximately by:
\begin{equation}
 E \left( \max_i \abs{ \omega_i - \frac{1}{2} } \right) \approx \sigma \sqrt{3} \left ( \frac{n-2}{n} \right ) = \frac{n-2}{2n} \label{neweq:23}
\end{equation}
This uses the fact that the expected maximum of $n$ independent uniform distributions is \cite{Kim1991}:
\[
\sigma \sqrt{3} \left( \frac{n-1}{n+1} \right)
\]
and the fact that the distributions are partially correlated, so that one of the $\omega_i$ can be inferred from $\bar{\omega}$ and the other $\omega_i$. In addition, the expected value of $\bar{\omega}$ differs from \onehalf\ by approximately:
\begin{equation}
E \left( \abs { \frac{1}{2} - \bar{\omega} } \right)  \approx \sigma \sqrt{\frac{2}{n \pi}} = \sqrt {\frac{1}{6 n \pi}} \label{neweq:24}
\end{equation}
By (\ref{neweq:22}), (\ref{neweq:23}), and (\ref{neweq:24}), a good estimate for the expected critical coupling $\ekcstar$ will be:
\begin{equation}
\ekcstar \approx \sigma \sqrt{3} \left ( \frac{n-2}{n} \right ) + \sigma \sqrt{\frac{2}{n \pi}} = \frac{n-2}{2n} + \sqrt {\frac{1}{6 n \pi}}
\end{equation}
Figure \ref{fig:star} shows that, for $n \geq 40$, this is an excellent estimate of the actual critical couplings, which are shown by solid triangles. It can also be seen that, as $n$ tends to infinity, the value of $\ekcstar$ approaches \onehalf. The expected critical coupling for the star acts as a lower bound for trees in general, since star trees have the smallest possible partitions.

\begin{figure}[htbp]
\begin{center}
\includegraphics[width=7.8cm]{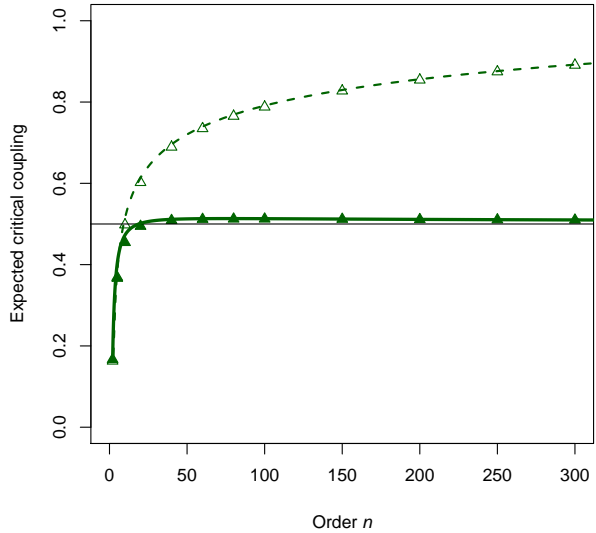}
\caption{Expected critical couplings for the star, as a function of the number of vertices $n$. The thicker solid line shows the curve $(n - 2) / (2 n) + \sqrt {1 / (6 n \pi)}$, while the thinner line shows the infinite-size limit of \onehalf\ for uniform distributions. For comparison, open triangles show expected critical couplings for normally distributed frequencies, and the dashed line shows the estimator $\sigma \mu(n-2)$. } \label{fig:star}
\end{center}
\end{figure}

\begin{figure}[htbp]
\begin{center}
\includegraphics[width=7.8cm]{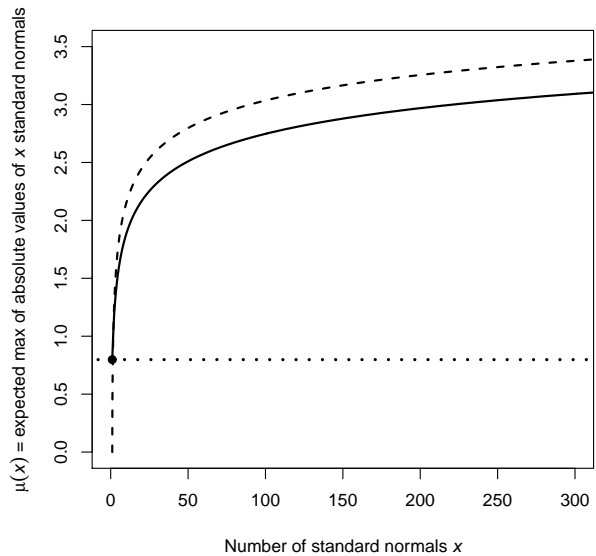}
\caption{The function $\mu(x)$, the expected value of the maximum of the absolute values of $x$ independent normally distributed random variables of unit variance (solid line). The dashed line shows the approximation  $\sqrt{2 \log (x)}$, while the dotted line shows $\mu(1) = \sqrt{2/\pi}$. } \label{fig:mu}
\end{center}
\end{figure}

For comparison, with normally distributed frequencies, the expected maximum of $\abs{\omega_i - \bar{\omega}}$ can be very closely approximated using the expected value of the maximum of the absolute values of $x$ independent normally distributed random variables of unit variance, which we write as $\mu(x)$.  For $x = 1$, $\mu(1) = \sqrt{2/\pi}$, as shown in Figure \ref{fig:mu}.  For $x \geq 2$, the value of $\mu(x)$ is slightly higher than (and converges to) the expected value of the maximum of $x$ independent normally distributed random variables of unit variance and zero mean, which can be roughly approximated by $\sqrt{2 \log (x)}$ \cite{Kim1991}.  The function $\mu(x)$ can be
expressed more precisely using the inverse complementary error function:
\begin{equation}
\mu(x) = \sqrt{2} \left[(1-\gamma)\ \erfcinv\left(\frac{1}{x}\right) + \gamma\ \erfcinv\left(\frac{1}{xe}\right) \right]
\end{equation}
where $\gamma = 0.57721566490\ldots$ is the Euler-Mascheroni Constant. However, we set $\mu(0) = 0$ by definition.

The expected critical coupling for normally distributed frequencies, shown by open triangles in Figure \ref{fig:star}, is very closely approximated by:
\[
\sigma \mu(n-2) = \frac{\mu(n-2)}{\sqrt{12}}
\]

It is interesting to compare this result with the work of Bronski \textit{et al.} \cite{Bronski2012}, who examine the case of the fully-connected network, in which every vertex is connected like the central vertex of the star. Their Theorem 4.1 shows that, with normally distributed frequencies, a scaling factor of:
\[
\frac{\sqrt{2 \log (n)}}{n+1}
\]
applies in the large-$n$ limit \cite{Bronski2012}, consistent with the approximation of $\mu(x)$ discussed above, and with the inverse relationship between $k_c$ and $n$ expressed in the usual scaling factor of $n^{-1}$ \cite{Arenas2008}.

\subsection{Dumb-bells} \label{ss:dumbell}

For ``dumb-bell'' trees, such as the one in Figure \ref{fig:examples} (b), the critical coupling is the result of the combined effects of the $n-2$ leaves and the fact that the tree can be partitioned into two halves. The latter has the greater influence, so by Theorem \ref{main:theorem}:
\begin{equation}
\kcdb \approx \abs{\sum_{i=1}^{n/2} (\omega_i - \bar{\omega})} = \abs{\sum_{i=1+n/2}^{n} (\omega_i - \bar{\omega})}
\end{equation}
That is, the critical coupling $\kcdb$ is the deviation from $\bar{\omega}$ for the average of the $\omega_i$ for one half of the dumb-bell (without loss of generality, we will use $\omega_1,\ldots,\omega_{n/2}$). To calculate $\ekcdb$, the expected value of this deviation, we note that:
\begin{equation}
\bar{\omega} = \frac{1}{n} \left( \sum_{i=1}^{n/2} \omega_i + \sum_{i=1+n/2}^{n} \omega_i \right)
\end{equation}
Thus:
\begin{equation}
\sum_{i=1}^{n/2} \omega_i  - \frac{n}{2} \bar{\omega} = \sum_{i=1}^{n/2} \omega_i  - \frac{1}{2} \sum_{i=1}^{n} \omega_i  = \frac{1}{2} \sum_{i=1}^{n/2} \omega_i - \frac{1}{2} \sum_{i=1+n/2}^{n} \omega_i  \label{tony:eq29}
\end{equation}
The variance for this difference of sums is $n\sigma^2/4$, and hence the expected critical coupling will be:
\begin{equation}
\ekcdb \approx \sqrt{\frac{2}{\pi}} \sqrt{\frac{n\sigma^2}{4}} = \sigma \sqrt {\frac{n}{2 \pi}} = \sqrt {\frac{n}{24 \pi}} \label{tony:eq30}
\end{equation}
The dashed line in Figure \ref{fig:bell} shows this approximation.  A better approximation can be obtained by noting that when calculating the critical coupling from $\abs{\omega_i - \bar{\omega}}$ for half the dumb-bell gives a result of less than  $1/2 = \sigma \sqrt{3}$, then the critical coupling will be determined by the effect of the leaves, in  a similar way to the expected value for the star. Since the variance for (\ref{tony:eq29}) is $n\sigma^2/4$, the standard deviation is:
\begin{equation}
\sigma' = \frac {\sigma}{2} \sqrt{n}
\end{equation}
Using a first-order approximation, the probability that a normal distribution with that standard deviation is within $1/2 = \sigma \sqrt{3}$ of the mean will be approximately:
\[
2 \frac{\sigma \sqrt{3}}{\sigma' \sqrt{2 \pi}} = 2 \sqrt {\frac{6}{n \pi}}
\]
The impact of the leaves is, with this probability, to replace an expected value of about  $1/4 = \sigma (\sqrt{3})/2$ by one of about  $1/2 = \sigma \sqrt{3}$, i.e.\ to increase the expected critical coupling to:
\begin{eqnarray}
\ekcdb & = & \sigma \sqrt{\frac{n}{2 \pi}} + 2 \sqrt {\frac{6}{n \pi}} \left (\sigma \frac{\sqrt{3}}{2} \right) \\
 & = & \sigma \sqrt{\frac{n}{2 \pi}} + \sigma \sqrt {\frac{18}{n \pi}} \;=\; \sqrt {\frac{n}{24 \pi}} + \sqrt{\frac{3}{2 n \pi}} \nonumber  
\end{eqnarray}
Figure \ref{fig:bell} shows that this correction ensures an excellent fit for $n \geq 20$. Note that as $n$ increases, the correction factor tends to zero, and $\ekcdb$ tends to the value given by (\ref{tony:eq30}).

\begin{figure}[htbp]
\begin{center}
\includegraphics[width=7.8cm]{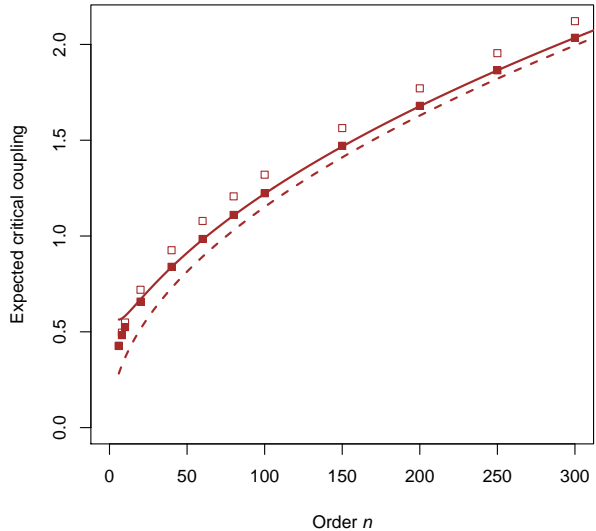}
\caption{Expected critical couplings for the dumb-bell. The dashed line shows the initial approximation $\sqrt{n / (24 \pi)}$ and the solid line the improved approximation, which adds the correction factor $\sqrt{3 / (2 n \pi)}$ due to the leaves. For comparison, open squares show values for normally distributed frequencies, which are higher due to an increased contribution from the leaves. } \label{fig:bell}
\end{center}
\end{figure}

\subsection{Binary trees}
We also considered binary trees, such as the one in Figure \ref{fig:examples} (c), i.e.\ trees with a root vertex $\nu_1$, and with vertices $\nu_{2i}$ and $\nu_{2i+1}$ having $\nu_i$ as a parent vertex. For these trees, the calculated expected critical couplings were well-predicted by the empirical curve $0.212 \sqrt{n} - 0.082$, shown in Figure \ref{fig:bintree}. A precise analytic solution for this case is difficult, since the binary tree contains partitions of multiple sizes, ranging from 1 to $(n-1)/2$. However, as Figure \ref{fig:bintree} shows, the binary tree is intermediate between the chain and the dumb-bell.

\begin{figure}[htbp]
\begin{center}
\includegraphics[width=7.8cm]{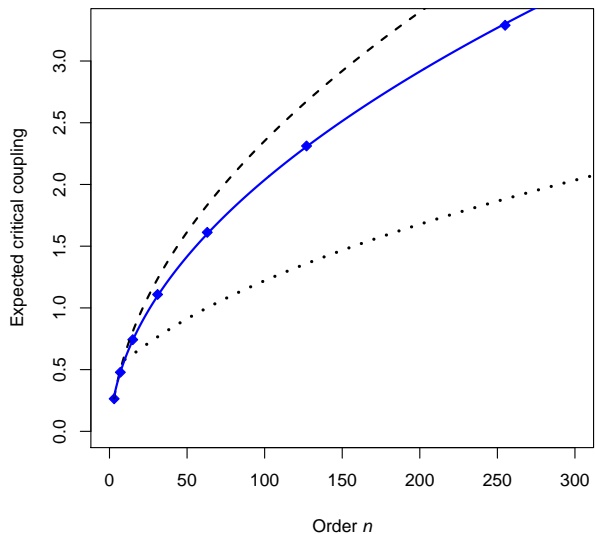}
\caption{Expected critical couplings for the binary tree. The solid line shows the empirical curve $0.212 \sqrt{n} - 0.082$. For comparison, the dashed line shows the empirical curve for the chain, while the dotted line shows the theoretical curve for the dumb-bell. Results for normally distributed frequencies are, as expected, virtually identical, and cannot be distinguished on this plot. } \label{fig:bintree}
\end{center}
\end{figure}

\subsection{Tadpoles}
``Tadpole'' graphs, such as the one in Figure \ref{fig:examples} (d), consist of a star combined with a ``tail,'' so that the diameter $D > 2$. Essentially, this is a kind of asymmetrical dumb-bell. The results in Figure \ref{fig:taddy} are based on tadpoles with $D = 8$.  For our uniformly distributed frequencies, the critical couplings converge to a limit of about 0.85, which reflects the combined contributions of the tail and the star.  Such a combination of effects is typical of non-symmetrical trees.  For normally distributed frequencies, the expected critical coupling grows logarithmically, as with the star.

\begin{figure}[htbp]
\begin{center}
\includegraphics[width=7.8cm]{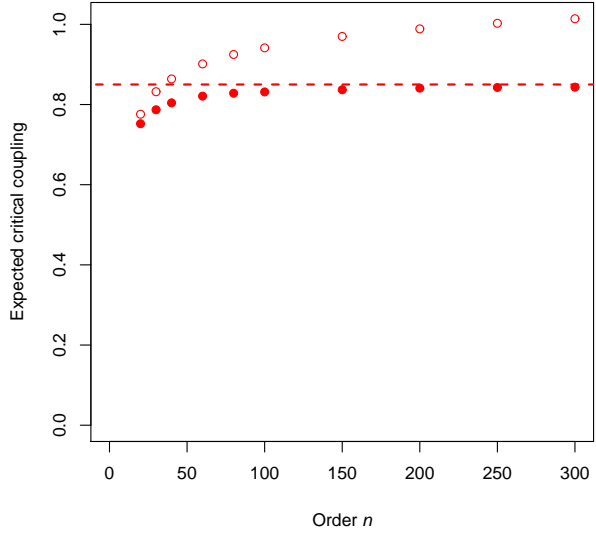}
\caption{Expected critical couplings for the $D=8$ tadpole, as a function of the number of vertices $n$. For our uniformly distributed frequencies, the critical couplings converge to a limit of about 0.85. For comparison, open circles show expected critical couplings for normally distributed frequencies. These grow logarithmically, as with the star. } \label{fig:taddy}
\end{center}
\end{figure}

\subsection{Upper and lower bounds}
\label{boundsec}

We now consider bounds on the expected critical couplings for trees with uniformly distributed frequencies. It is helpful to define $P_e$ as the number of vertices in the smaller of the two sub-trees $T_e$ and $T'_e$ formed when $e$ is deleted from $T$, and to define $P$ as the maximum of the $P_e$. We call $P$ the \textit{maximum partition size}.

In terms of $n$, the chain provides an upper bound for expected critical couplings on trees, since rearranging the topology of a chain can only result in some or all of the $P_e$ being smaller. In terms of diameter, however, the chain provides a lower bound, as Figure \ref{fig:combod} illustrates. This follows from considering the longest subchain (with $D+1$ vertices) within a tree. Any additional vertices can only increase the expected critical coupling. Consequently, the critical coupling for a tree of diameter $D$ is bounded below by the critical coupling for a chain of diameter $D$. Since $\chi(x)$ gives the expected critical coupling for a chain of $x$ vertices, we have:
\begin{equation}
\chi(D+1) \leq \ekc \leq \chi(n)    \label{lbeq:d}
\end{equation}
For example, for a tadpole tree with $D = 8$, $n \geq 17$, and our choice of frequencies, $\ekc$ tends to 0.85, and:
\[
\chi(9) = 0.59 < 0.85 <  0.87 = \chi(17)
\]

\begin{figure}[htbp]
\begin{center}
\includegraphics[width=7.8cm]{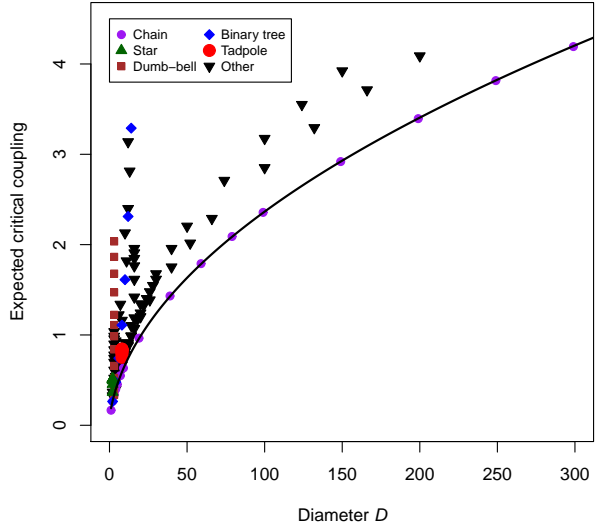}
\caption{Expected critical couplings for trees in terms of the diameter $D$. Downward-pointing black triangles represent trees not discussed above (including Y-shaped trees, X-shaped trees, and Scale-Free trees). The solid line gives the empirical curve $\chi(D+1)$ for the chain. } \label{fig:combod}
\end{center}
\end{figure}

The maximum partition size also produces a lower bound. While the chain and the dumb-bell have different diameters, they both have $P = \floor{n/2}$. For the binary tree, $P = \floor{n/2}$, while for the star, $P = 1$. For tadpole trees, $P = D-1$, provided $n \geq 2(D-1)$.

Partitions of size at least $P$ must occur at least twice, and these partitions are not independent, since the $\omega_i$ sum to $n\bar{\omega}$. Suppose that we decompose the tree into $r$ partitions (each of the form $T_e$ for some edge $e$), that $X_j$ is the sum of the $\omega_i$ over the $j^{\mathrm{th}}$ partition, and that the $j^{\mathrm{th}}$ partition contains $m_j$ vertices. Then:
\begin{equation}
X_j - m_j \bar{\omega} = (1 - \frac{1}{r})X_j - \frac{1}{r} \sum_{i \neq j} X_j
\end{equation}
Since we are considering partitions of size at least $P$, by a standard result on the variance of linear forms in distributions \cite{Rohatgi2001}, $X_j - m_j \bar{\omega}$ has a standard deviation of at least:
\begin{equation}
\sigma'' = \sigma \sqrt{P} \sqrt{ \left( \frac{r-1}{r} \right)^{2}+(r-1)\frac{1}{r^2} } = \sigma \sqrt{P} \sqrt{1-\frac{1}{r}} \label{richard:r}
\end{equation}
Since the number of partitions $r \geq 2$, we obtain a lower bound on $\ekc$:
\begin{equation}
\ekc \geq \sigma \sqrt{\frac{2}{\pi}} \sqrt {\frac{P}{2}} \geq \sigma \sqrt{\frac{P}{\pi}} = \sqrt {\frac{P}{12 \pi}}   \label{lbeq:p}
\end{equation}
Figure \ref{fig:part} illustrates this lower bound, using the same datapoints as Figure \ref{fig:combod}. For trees with small diameter, like the dumb-bell, this is tighter than the bound based on $D$.

\begin{figure}[htbp]
\begin{center}
\includegraphics[width=7.8cm]{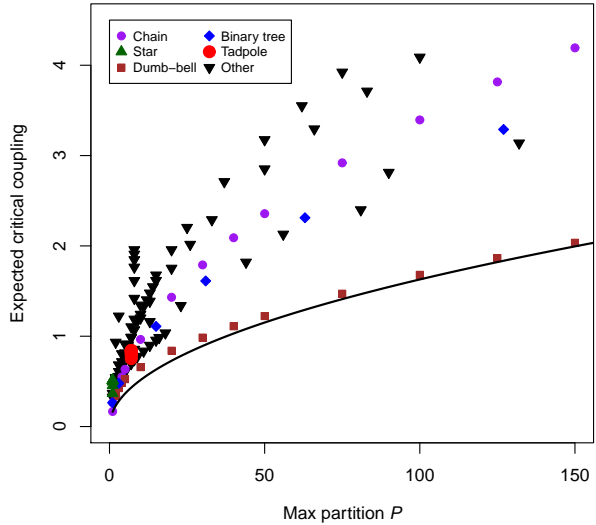}
\caption{Expected critical couplings for trees in terms of the maximum partition size $P$. The solid line gives the lower bound $\sqrt{P / (12 \pi)}$. Datapoints are the same as Figure \ref{fig:combod}.} \label{fig:part}
\end{center}
\end{figure}

We also obtain an upper bound in terms of $P$ by noting that the worst-case partition is half of a chain of length $2 P$, and that this partition occurs at most $n/P$ times. We can approximate $\mu(n/P)$ by $\sqrt{2 \log (n/P)}$, so that, for suitable constants $a$, $b$, and $c$:
\begin{equation}
\ekc \leq a \chi(2 P) \sqrt{2 \log \frac{n}{P}} = b \sigma \sqrt{P} \sqrt{2 \log \frac{n}{P}} = c \sigma \sqrt{P \log \frac{n}{P}}
\end{equation}
Since the chain is the worst case, from (\ref{chain:chidef}) it suffices to take $c = \frac{3}{2} > 0.873 \sqrt {2/ \log 2}$, i.e.
\begin{equation}
\ekc < \frac{3}{2} \sigma \sqrt{P \log \frac{n}{P}} = \frac{1}{4} \sqrt{3 P \log \frac{n}{P}}  \label{ubeq:p}
\end{equation}
Figure \ref{fig:mult} illustrates this bound.

\begin{figure}[htbp]
\begin{center}
\includegraphics[width=7.8cm]{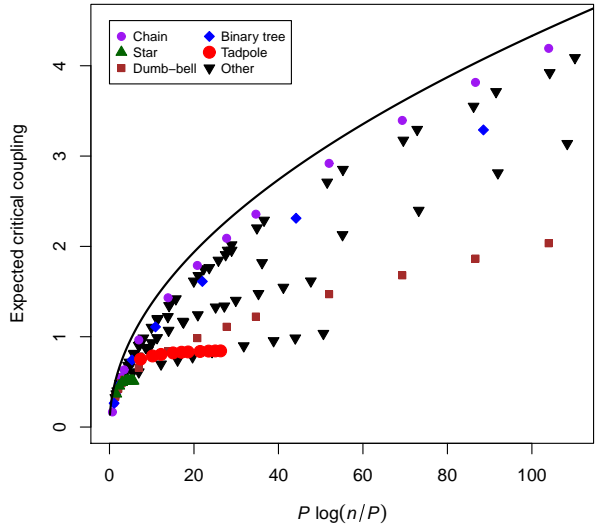}
\caption{Expected critical couplings for trees in terms of $P \log (n/P)$. The solid line gives the upper bound $(1/4) \sqrt{3 P \log (n/P)}$. Datapoints are the same as in Figure \ref{fig:combod}.} \label{fig:mult}
\end{center}
\end{figure}

\subsection{Bounds for normally distributed frequencies}
Since normally distributed frequencies lead to critical couplings that are either virtually identical to or larger than the critical couplings for uniformly distributed frequencies, and since the chain continues to be the worst case, the bounds in the previous section apply also for normally distributed frequencies. In particular, Figure \ref{fig:mnorm} confirms visually that the upper bound (\ref{ubeq:p}) still holds.

\begin{figure}[htbp]
\begin{center}
\includegraphics[width=7.8cm]{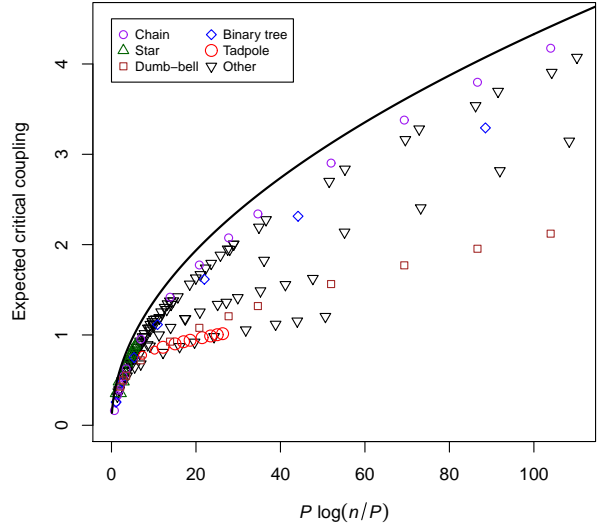}
\caption{Expected critical couplings for trees with normally distributed frequencies in terms of $P \log (n/P)$. The solid line gives the upper bound $(1/4) \sqrt{3 P \log (n/P)}$, as in Figure \ref{fig:mult}.} \label{fig:mnorm}
\end{center}
\end{figure}

In terms of the order $n$, expected critical couplings are bounded above by $\chi(n)$, the value for the chain, as is the case for the uniform distribution. Expected critical couplings are bounded below by $\sigma \mu(n-2)$, the value for the star, which grows logarithmically. Figure \ref{fig:newnorm} illustrates these two bounds. It follows from these bounds that, for all trees with normally distributed frequencies, $k_c \rightarrow \infty$ as $n \rightarrow \infty$. This is consistent with Corollary 3.2 of \cite{Strogatz1998}, that for trees with ``broad-banded'' frequency distributions, the probability of a phase-locked fixed point tends to zero.

\begin{figure}[htbp]
\begin{center}
\includegraphics[width=7.8cm]{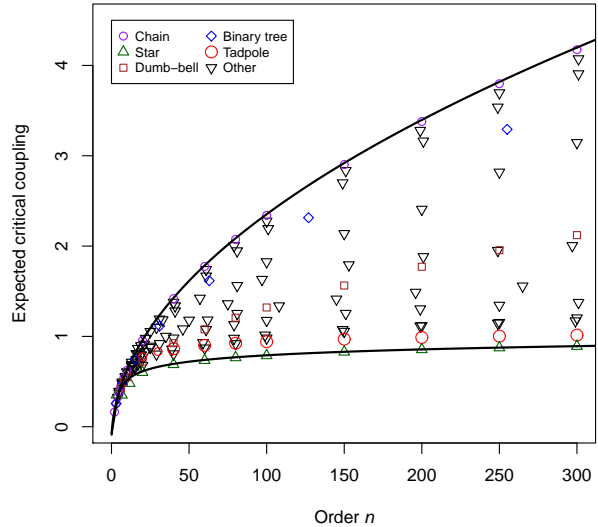}
\caption{Expected critical couplings for trees with normally distributed frequencies in terms of the order $n$. The solid lines gives the upper bound $\chi(n)$ and the lower bound $\sigma \mu(n-2)$. Note that, in every case, $k_c \rightarrow \infty$ as $n \rightarrow \infty$.} \label{fig:newnorm}
\end{center}
\end{figure}


\section{Minimising the Critical Coupling using Rearrangement}
\label{sec:rearrange}

The results above have application to organisational trees.  In this case, the frequencies $\omega_i$ can reflect the speed of the decision cycle for person $i$ (as in \cite{Kalloniatis2008}) or, more abstractly, they can represent a one-dimensional projection of the activities of person $i$ (as in \cite{Dekker2011}).
In both cases, synchronisation problems arise when vertices with different $\omega_i$ are not strongly coupled. Organisational structures are often hierarchical trees (e.g.\ binary trees), because such trees offer a number of advantages, such as low diameter (worst-case distance between pairs of vertices) together with low degree (the number of subordinates an individual has responsibility for) \cite{roux2002}. If we also consider synchronisation to be favourable (in that steady organisational work patterns are achieved), we would like to combine an efficient tree structure with a low critical coupling. 

As an example, Figure \ref{fig:org} shows the structure of a real organisation, where the numbers in each vertex are a one-dimensional projection of each person's activities, calculated by administering a survey and taking the most informative of the principal components of the answers.
Interpreting these numbers as frequencies, using Theorem \ref{main:theorem} we calculate a critical coupling $k_c = 5.08$.  Improving the synchronisation of the network of people within this organisation requires either management activities to more strongly couple people's work, or some form of structural reorganisation.

\begin{figure}[htbp]
\begin{center}
\includegraphics[width=7.8cm]{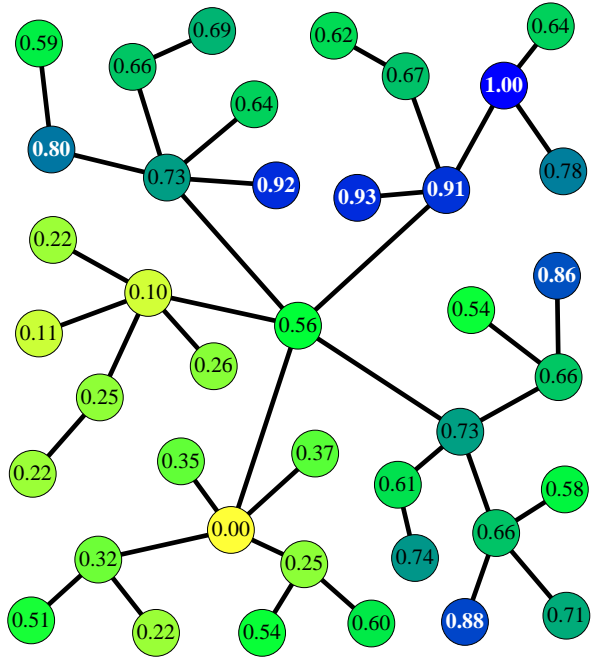}
\caption{An organisational tree with a critical coupling $k_c = 5.08$, which reduces to 1.07 after rearranging frequencies. The two branches at the lower left are somewhat out of step with the rest of the organisation.} \label{fig:org}
\end{center}
\end{figure}

As an example of structural reorganisation, for the organisation in Figure \ref{fig:org}, we can construct a Kuramoto tree which synchronises much more easily (with $k_c = 1.07$) by leaving the topology of the organisational tree unchanged, but re-shuffling people (i.e.\ frequencies $\omega_i$) within it.  In a managed situation like an organisation, we often have some discretion to reshuffle people in this way.  In fact, the following theorem shows that we can reduce the critical coupling to a value which is independent of the number of vertices:

\begin{theorem}\label{thm:tree}Let $T$ be any tree network of $n\geq 2$ vertices, and $\omega_{i}$, for $i = 1,\ldots,n$, be the collection of natural frequencies. Let $T^{a}$ be the tree which is isomorphic to $T$ in which the frequencies are assigned to vertices according to some assignment $a$. Then for some $a$,
\begin{equation}
k_{c}(T^{a}) \leq \omega_{\max}-\omega_{\min}
\end{equation}
where $\omega_{\max}$ is the largest frequency and $\omega_{\min}$ is the smallest.
\end{theorem}

\begin{proof}
The proof is constructive. First we note that by traversing the edges and vertices of the tree in depth-first order, the corresponding assignment of the vertices of $T$ to the natural numbers $1,2,\ldots,n$ has the property that for any edge $e$, at least one of the components of $T-{e}$ is assigned a set of numbers that are of the contiguous form $k,k+1,\dots,m$ for some $k$ and $m$. In Figure \ref{rtfigure2} we show an example of such a vertex assignment giving the correspondence between edges and contiguous number sets.

\begin{figure}[htbp]
\begin{center}
\includegraphics[width=7.8cm]{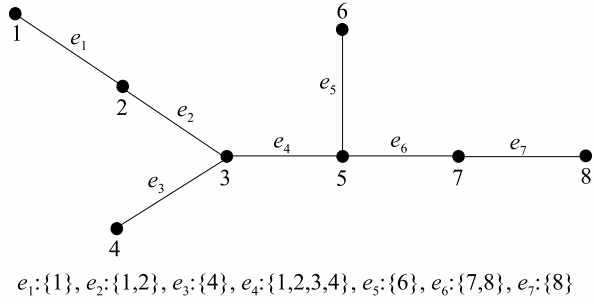}
\caption{An assignment of the numbers $1,\ldots,8$ to vertices of a tree such that, for any edge $e$, at least one of the components of $T-{e}$ is assigned a contiguous set of numbers, as per Theorem \ref{thm:tree}.}
\label{rtfigure2}
\end{center}
\end{figure}

We shall assume that we have such an assignment of numbers to vertices and that vertex $\nu_i$ refers to the vertex assigned to the number $i$. We now order the natural frequencies as follows. Select the first frequency to be $\omega_{\max}$. Add frequencies $\omega_{a1},\omega_{a2},\ldots,\omega_{ap}$ from any previously unselected frequencies less than $\bar{\omega}$, stopping at the least $p$ where:
\begin{equation}
\sum_{i=1}^p \omega_{ai} \leq p\,\bar{\omega}
\end{equation}
Add frequencies $\omega_{a(p+1)},\omega_{a(p+2)},\ldots,\omega_{aq}$ from any previously unselected frequencies greater than $\bar{\omega}$ stopping at the least $q$ where:
\begin{equation}
\sum_{i=1}^q \omega_{ai} \geq q\,\bar{\omega}
\end{equation}
In similar fashion continue to add vertices alternating between runs of frequencies less than $\bar{\omega}$ and greater than $\bar{\omega}$, until all frequencies are used. The function $f(k)$ given by:
\begin{equation}
f(k)=\left[\sum_{i=1}^k \omega_{ai}\right]-k\bar{\omega}
\end{equation}
must then take values between $\omega_{\max}-\bar{\omega}$ and $\omega_{\min}-\bar{\omega}$, with $f(n)=0$ (see Figure \ref{rtfigure3}).

\begin{figure}[htbp]
\begin{center}
\includegraphics[width=7.8cm]{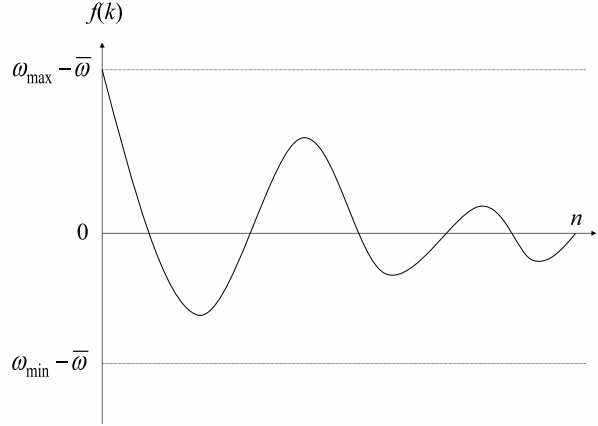}
\caption{Illustration of the function $f(k)$ from the proof of Theorem \ref{thm:tree}.}
\label{rtfigure3}
\end{center}
\end{figure}

Finally we assign frequency $\omega_{ai}$ to vertex $\nu_i$ for $i = 1,\ldots,n$. Recall that by our earlier observation, for any edge $e$ at least one of the components of $T-\{e\}$, say $T_{e}$, has vertices with natural frequencies from a contiguous sequence $\omega_{m},\omega_{m+1},\ldots,\omega_{m+r}$ for some $m$ and $r$. These frequencies correspond to a segment of the function $f$ from $f(m)$ to $f(m+r)$. Clearly:
\begin{equation}
| f(m+r)-f(m)| \leq (\omega_{\max}-\bar{\omega})- (\omega_{min}-\bar{\omega}) = \omega_{\max}-\omega_{\min} \label{rt:11}
\end{equation}
Also:
\begin{equation}
| f(m+r)-f(m)| = \left|\sum_{i=m}^{m+r}\omega_{ai}-(r+1)\bar{\omega}\right| = \left| \sum_{i \in {T_{e}}}\omega_{i}-|T_{e}| \bar{\omega} \right| \label{rt:12}
\end{equation}
Combining (\ref{rt:11}) and (\ref{rt:12}),
\begin{equation}
\left| \sum_{i \in {T_{e}}}\omega_{i}-|T_{e}| \bar{\omega} \right| \leq \omega_{\max}-\omega_{\min}
\end{equation}
Since this is true for all edges $e$, by Theorem \ref{main:theorem} we obtain the result. \qquad \end{proof}

The following example shows that this result is in a certain sense best possible. Let a star tree have 2 vertices with frequency $\xi$ and $n-2$ vertices with frequency $\zeta$. Then for every assignment $a$ of frequencies to vertices,
\begin{equation}
k_{c}(T^{a})=\left|\xi-\zeta+\frac{2\zeta-2\xi}{n}\right|\rightarrow\left|\xi-\zeta\right|
\end{equation}
This follows from Theorem \ref{main:theorem}, noting that every leaf vertex of the star must be assigned frequency $\xi$ or $\zeta$.

Figure \ref{fig:rearrange} shows an example for a 15-vertex binary tree.  
As another example, for 255-vertex binary trees, the expected critical coupling without rearrangement was 3.289, using 1,000,000 different frequency assignments, with all frequencies uniformly distributed over the interval [0,1] (i.e.\ $\omega_{\max} - \omega_{\min} \leq 1$). Rearrangement ensured that all 1,000,000 critical couplings were restricted to the range 0.462 to 0.960 (less than $\omega_{\max} - \omega_{\min}$), with an expected critical coupling of 0.685.

\begin{figure}[htbp]
\begin{center}
\includegraphics[width=7.8cm]{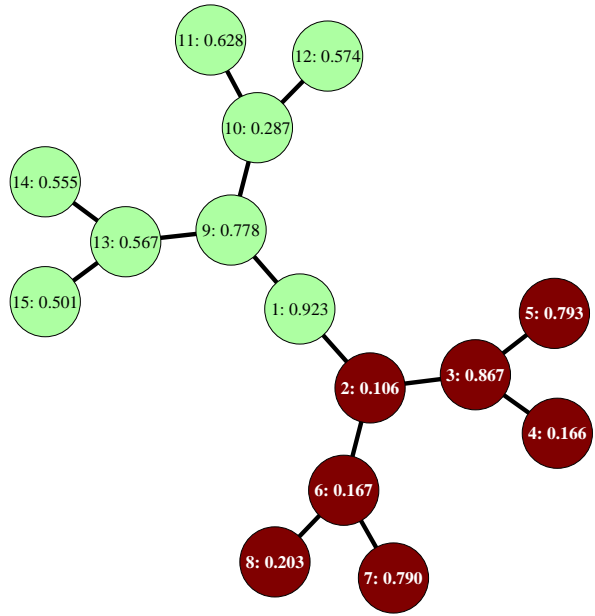}
\caption{A binary tree with frequencies uniformly distributed over the interval [0,1], after rearrangement. Vertices are labelled with depth-first-search index and frequency $\omega_i$. The mean frequency $\bar{\omega}$ is 0.527, and the critical coupling of 0.596 is based on the partition on the lower right.} \label{fig:rearrange}
\end{center}
\end{figure}


\section{Discussion}
\label{sec:disc}

In this paper we have explored the critical coupling for Kuramoto oscillators arranged in tree topologies.  With Theorem \ref{main:theorem}, we provided a closed-form solution (\ref{maineq:15}) for the critical coupling $k_c$, in the case that frequencies were known.

In the case that only frequency distributions were known, we calculated $\ekc$, the expected value of the critical coupling, for chains and stars (for both uniform and Gaussian vertex frequency distributions) and for ``dumb-bells'' (for uniform distributions only). We also provided empirical values of $\ekc$ for these cases, and for the ``tadpole'' and binary tree. For trees in general, we provided lower bounds on $\ekc$ in terms of the diameter $D$ (\ref{lbeq:d}) and the maximum partition size $P$ (\ref{lbeq:p}). We provided upper bounds in terms of the number of vertices (the expected critical coupling for the chain performs this function) and the maximum partition size (\ref{ubeq:p}). These bounds hold for both uniform and Gaussian frequency distributions.

Finally, Theorem \ref{thm:tree} showed that, for a given set of vertex frequencies, there is a rearrangement of oscillator frequencies for which the critical coupling is bounded by the spread of frequencies.

In future work, we hope to provide closed-form expressions for the expected critical coupling $\ekc$ for a number of other cases of tree and, more importantly, to use methods similar to Lemma \ref{main:lemma} to determine the critical coupling $k_c$ for other classes of graph. Corollary \ref{bell:corollary} is a small step in this direction.

\bibliography{trees7}
\bibliographystyle{siam}

\end{document}